\documentclass[11pt,a4paper]{article}

\usepackage{amsmath,amssymb,amsthm}
\usepackage[utf8]{inputenc}
\usepackage{mathrsfs}
\usepackage[margin=1in]{geometry}
\usepackage{hyperref}
\usepackage{enumitem}

\theoremstyle{plain}
\newtheorem{theorem}{Theorem}[section]
\newtheorem{proposition}[theorem]{Proposition}
\newtheorem{lemma}[theorem]{Lemma}
\newtheorem{corollary}[theorem]{Corollary}
\newtheorem{conjecture}[theorem]{Conjecture}

\theoremstyle{definition}
\newtheorem{definition}[theorem]{Definition}

\theoremstyle{remark}
\newtheorem{remark}[theorem]{Remark}

\newcommand{\E}{\mathbb{E}}
\newcommand{\R}{\mathbb{R}}
\newcommand{\Z}{\mathbb{Z}}
\newcommand{\C}{\mathbb{C}}
\newcommand{\T}{\mathbb{T}}

\newcommand{\Var}{\mathrm{Var}}

\title{Bakry--\'Emery Curvature of the Fractional Laplacian\\
via Fractional Brownian Covariance}

\author{Ramiro Fontes\\
Quijotic Research\\
\texttt{ramirofontes@gmail.com}}

\date{February 2026}

\begin{document}

\maketitle

\begin{abstract}
We establish the first positive Bakry--\'Emery curvature bound for
a fractional Laplacian---with and without drift---on any domain,
resolving a case of the open problem of extending the $\Gamma_2$-calculus
to non-local operators (Garofalo 2019, Spener--Weber--Zacher 2020).

Our starting point is the identification
$\Psi_\gamma(\xi,\eta) = R_{\gamma/2}(|\xi|,|\eta|)$
for same-sign frequencies $\xi\eta \geq 0$:
the Fourier-space carr\'e du champ of the $\gamma$-stable generator
$L_\gamma = -(-\Delta)^{\gamma/2}$ coincides with the covariance
kernel of fractional Brownian motion with Hurst parameter $H = \gamma/2$.
On $\T = \R/(2\pi\Z)$, this reduces the Bakry--\'Emery curvature
to a generalized eigenvalue problem for fBM covariance matrices.
At $\gamma = 1$ (Cauchy process, $H = 1/2$), these eigenvalues are the
odd integers $\{1,3,\ldots,2N{-}1\}$, giving $\mathrm{CD}(1,\infty)$
for all $N$ and hence on the full domain.
On real trigonometric polynomials, the cross-sign kernel
$\Psi_\gamma(n,-m)$ vanishes if and only if $\gamma = 1$,
making the Cauchy process the unique stability index with $\kappa \geq 1$.

For the L\'evy--Fokker--Planck operator
$L = -(-\Delta)^{1/2} - V'(x)\,\partial_x$ with confining potential
$V(x) = -\omega^2\cos x$, we prove that the drift correction acts
as a \emph{scalar shift} of the curvature spectrum: the eigenvalues
become $\{2k{-}1 + \frac{\omega^2}{2}\cos x\}_{k=1}^N$,
giving the global bound $\mathrm{CD}(1-\omega^2/2,\,\infty)$
for $\omega^2 < 2$, independent of $N$.
This yields a Poincar\'e inequality and gradient estimate for an
operator that is \emph{not} Fourier-diagonal, and whose spectrum
is not explicitly known.
These results complement Spener--Weber--Zacher's negative result
that $(-\Delta)^{\gamma/2}$ on $\R^d$ fails
$\mathrm{CD}(\kappa,N)$ for all finite~$N$.
\end{abstract}


\section{Introduction}\label{sec:intro}

\subsection{Motivation}

The Bakry--\'Emery $\Gamma_2$-criterion is the principal tool for
establishing functional inequalities for Markov
semigroups~\cite{BakryGentilLedoux2014}. For diffusion generators
$L = \Delta + \nabla V \cdot \nabla$ on a Riemannian manifold,
the condition $\Gamma_2(f,f) \geq \kappa\,\Gamma(f,f)$ is equivalent
to a Ricci curvature lower bound, and the constant $\kappa$ controls
convergence to equilibrium through
\[
  \Gamma_2 \geq \kappa\,\Gamma
  \;\Longrightarrow\;
  \text{log-Sobolev with } C = 2/\kappa
  \;\Longrightarrow\;
  \text{Poincar\'e with } \lambda_1 \geq \kappa.
\]
For non-local generators such as the fractional Laplacian
$L_\gamma = -(-\Delta)^{\gamma/2}$, the theory is substantially
more delicate. Spener, Weber, and
Zacher~\cite{SpenerWeberZacher2020} proved that the fractional
Laplacian on $\R^d$ fails to satisfy the curvature-dimension inequality
$\mathrm{CD}(\kappa,N)$---that is,
$\Gamma_2(f,f) \geq \kappa\,\Gamma(f,f)
+ \tfrac{1}{N}(Lf)^2$---for all $\kappa \in \R$ and all finite $N > 0$.
(Throughout this paper we use the condition
$\mathrm{CD}(\kappa,\infty)$, meaning
$\Gamma_2 \geq \kappa\,\Gamma$ with no dimension term.)
This was identified as a key open problem in extending the $\Gamma$-calculus
to non-local operators~\cite{Garofalo2019}.

In this paper, we work on the compact torus $\T = \R/(2\pi\Z)$, where
the spectral gap $\lambda_1 = 1$ provides the essential compactness that
the $\R^d$ setting lacks.  Our starting point is the observation that
the Fourier-space carr\'e du champ of the $\gamma$-stable generator
coincides with the covariance kernel of fractional Brownian motion
with Hurst parameter $H = \gamma/2$.
This identification---a connection between the Dirichlet form of a
jump process and the covariance of a Gaussian process---allows us to
express the Bakry--\'Emery curvature theory on same-sign frequencies
in terms of spectral properties of fBM covariance matrices.

\subsection{The stable--fBM correspondence}

The Fourier-space kernel of the carr\'e du champ
$\Psi_\gamma(\xi,\eta) :=
\frac{1}{2}(|\xi|^\gamma + |\eta|^\gamma - |\xi-\eta|^\gamma)$
is a standard object in the theory of jump-type Dirichlet
forms~\cite{BoettcherSchillingWang2013}.
The fBM covariance $R_H(s,t) = \frac{1}{2}(|s|^{2H} + |t|^{2H}
- |s-t|^{2H})$ is classical since
Mandelbrot--Van~Ness~\cite{Mishura2008}.
The kernel $\frac{1}{2}(|\xi|^\gamma + |\eta|^\gamma
- |\xi-\eta|^\gamma)$ also appears in Schoenberg's
theory of conditionally negative definite
functions~\cite{Schoenberg1938}.

These three appearances have not, to our knowledge, been connected
explicitly.  We prove (Theorem~\ref{thm:duality},
Section~\ref{sec:duality}):
\begin{equation}\label{eq:duality-intro}
  \Psi_\gamma(\xi,\eta) = R_{\gamma/2}(|\xi|,|\eta|),
  \qquad \gamma \in (0,2),\quad \xi\eta \geq 0,
\end{equation}
identifying the carr\'e du champ of the stable generator on
same-sign frequencies with the covariance of fBM.  The proof
requires only the L\'evy--Khintchine formula.
The identity was originally discovered within an operator
factorization framework for stochastic
calculus~\cite{Fontes2026a}; here we give a self-contained
treatment and develop its consequences for curvature theory.

\begin{remark}[On the depth of the identification]
\label{rem:identification}
The identity~\eqref{eq:duality-intro} (on same-sign frequencies)
can be verified
in one line once both sides are written out. Its value
is not computational but \emph{structural}: it connects
the Dirichlet form of a pure-jump process to the
covariance of a Gaussian process, placing the entire
Bakry--\'Emery theory into the well-developed spectral
theory of fBM covariance operators. Concretely:
\begin{enumerate}[label=(\alph*)]
\item The positive semi-definiteness of $\Psi_\gamma$
  follows immediately from the positive-definiteness
  of $R_H$ (Schoenberg's theorem becomes a
  corollary of fBM existence).
\item The Hadamard square identity (Theorem~A)
  translates $\Gamma_2 \geq 0$ into the Schur product
  theorem applied to fBM covariance matrices.
\item The curvature computation reduces to a
  \emph{generalized eigenvalue problem} for fBM
  matrices, making numerical computation trivial
  and structural analysis (e.g., the odd integer
  spectrum at $\gamma = 1$) accessible.
\item The cross-sign vanishing (Theorem~D) is
  equivalent to the linearity of the fBM path
  variance $|s|^{2H} + |t|^{2H} - |s+t|^{2H} = 0$
  at $H = 1/2$.
\end{enumerate}
None of these consequences are obvious from the
Dirichlet form or L\'evy--Khintchine perspectives alone.
\end{remark}

\subsection{Main results}

We work on $\T = \R/(2\pi\Z)$ with normalized Lebesgue measure
$\mu = dx/(2\pi)$. The fractional Laplacian acts on Fourier modes as
$L_\gamma(e^{inx}) = -|n|^\gamma\,e^{inx}$.

\medskip\noindent
\textbf{Theorem~A} (Hadamard square identity; \S\ref{sec:gamma2}).
\emph{The iterated carr\'e du champ kernel of $L_\gamma$ is the
Hadamard square of the carr\'e du champ kernel:
\begin{equation}\label{eq:intro-A}
  \Phi_{\gamma,2}(\xi,\eta)
  = \bigl[\Psi_\gamma(\xi,\eta)\bigr]^2.
\end{equation}
On same-sign frequencies ($\xi\eta \geq 0$), this equals
$[R_{\gamma/2}(|\xi|,|\eta|)]^2$.}

\medskip\noindent
\textbf{Theorem~B} (Curvature as eigenvalue problem; \S\ref{sec:curv-pos}).
\emph{On positive-frequency polynomials $f = \sum_{n=1}^N a_n\,e^{inx}$,
the Bakry--\'Emery curvature is
\begin{equation}\label{eq:intro-B}
  \kappa(\gamma, N) = \lambda_{\min}\!\bigl(R_H^{\circ 2},\, R_H\bigr),
  \quad H = \gamma/2,
\end{equation}
where $R_H$ is the $N \times N$ fBM covariance matrix on $\{1,\ldots,N\}$
and $R_H^{\circ 2}$ is its entrywise square (Hadamard square;
see \S\ref{sec:notation}).}

\medskip\noindent
\textbf{Theorem~C} (Odd integer eigenvalues; \S\ref{sec:cauchy}).
\emph{At $\gamma = 1$, the matrix $R_{1/2}^{-1}\,R_{1/2}^{\circ 2}$ is
upper triangular with eigenvalues $\{1, 3, 5, \ldots, 2N-1\}$.
In particular, $\kappa(1,N) = 1$ for all $N \geq 1$.}

\medskip\noindent
\textbf{Theorem~D} (Cross-sign vanishing and uniqueness of $\gamma=1$; \S\ref{sec:cross-sign}).
\emph{The cross-sign kernel satisfies
\begin{equation}\label{eq:intro-D}
  \Psi_\gamma(n,-m) = \tfrac{1}{2}\bigl(|n|^\gamma + |m|^\gamma
  - |n+m|^\gamma\bigr)
  \begin{cases}
    = 0 & \text{if } \gamma = 1, \\
    > 0 & \text{if } \gamma < 1, \\
    < 0 & \text{if } \gamma > 1,
  \end{cases}
\end{equation}
for all $n,m > 0$. Consequently, $\gamma = 1$ is the unique stability
index for which the curvature
$\kappa_{\mathrm{real}}(\gamma) \geq 1$ on all real
trigonometric polynomials (where $\kappa_{\mathrm{real}}$ denotes the
Bakry--\'Emery curvature restricted to real-valued test functions).}

\medskip\noindent
\textbf{Theorem~E} (Global curvature under drift at $\gamma=1$; \S\ref{sec:potential}).
\emph{For $L = -(-\Delta)^{1/2} - V'(x)\,\partial_x$ on $\T$ with
$V(x) = -\omega^2\cos x$, the drift correction to $\Gamma_2$ satisfies
the scalar identity
\begin{equation}\label{eq:intro-drift-scalar}
  \Gamma_{L,2}(f,f)(x) - \Gamma_{1,2}(f,f)(x)
  = \frac{\omega^2}{2}\cos(x)\;\Gamma(f,f)(x)
\end{equation}
for all positive-frequency polynomials
$f \in \mathcal{T}_N^+$ (\S\ref{sec:prelim}) and all $x \in \T$.
The eigenvalues of the curvature operator become
$\kappa_k(x) = (2k{-}1) + \frac{\omega^2}{2}\cos x$, giving
\begin{equation}\label{eq:intro-E}
  \kappa(1,\omega^2) = 1 - \frac{\omega^2}{2},
\end{equation}
independent of $N$.  In particular,
$\mathrm{CD}(1-\omega^2/2,\,\infty)$ holds for $\omega^2 < 2$.}

\subsection{Context and prior work}

This paper provides, to our knowledge, the first positive
Bakry--\'Emery curvature result for the fractional Laplacian
on any domain---both without drift ($\mathrm{CD}(1,\infty)$
for the free Cauchy process, Theorem~C) and with drift
($\mathrm{CD}(1-\omega^2/2,\infty)$ for the
Cauchy--Fokker--Planck operator, Theorem~E).
The drift result is the paper's principal contribution:
the operator $L = -(-\Delta)^{1/2} - \omega^2\sin(x)\,\partial_x$
is not diagonalized by Fourier modes, and its spectrum
is not explicitly known.  The curvature bound yields
a Poincar\'e inequality and gradient estimate that cannot
be obtained by direct spectral methods.

Spener, Weber, and Zacher~\cite{SpenerWeberZacher2020} showed that
$(-\Delta)^{\gamma/2}$ on $\R^d$ fails the curvature-dimension inequality
$\mathrm{CD}(\kappa,N)$ for all $\kappa$ and all finite $N$. Their result
leaves open the case of compact domains. Our Theorem~C provides
$\mathrm{CD}(1,\infty)$ on the torus at $\gamma = 1$: the compactness
(spectral gap $\lambda_1 = 1$) is essential.

Gentil and Imbert~\cite{GentilImbert2008} prove exponential convergence
of $\Phi$-entropies for the L\'evy--Fokker--Planck equation with rate
$\gamma$ for $\gamma$-stable processes with Ornstein--Uhlenbeck drift,
bypassing the chain rule via entropy production. Their approach gives
existence of convergence but not explicit Bakry--\'Emery constants.
Our results are complementary: we work on the torus rather than $\R^d$
but obtain exact curvature formulas controlled by the fBM covariance.

Chafa\"i and Malrieu (unpublished, cited in~\cite{GentilImbert2008},
Remark~4) observe that $\Gamma_2 \geq 0$ for pure-jump generators
follows from a double-integral representation as a square. We go
much further: we compute the \emph{full sesquilinear kernel}, identify
it as the Hadamard square of the fBM covariance, and extract
quantitative curvature bounds.

\subsection{Organization}

Section~\ref{sec:prelim}: preliminaries on the fractional Laplacian,
carr\'e du champ, and fBM.
Section~\ref{sec:duality}: the stable--fBM correspondence (identity
$\Psi_\gamma = R_{\gamma/2}$ on same-sign frequencies) and the distinguished point $\gamma = 1$.
Section~\ref{sec:gamma2}: Hadamard square identity (Theorem~A).
Section~\ref{sec:curvature}: curvature on positive frequencies
(Theorems~B and~C) and the phase transition.
Section~\ref{sec:real}: cross-sign structure and real polynomials
(Theorem~D).
Section~\ref{sec:potential}: curvature with confining potential
(Theorem~E).
Section~\ref{sec:func-ineq}: functional inequality consequences.
Section~\ref{sec:discussion}: spectral invariants and open problems.

\section{Preliminaries}\label{sec:prelim}

\subsection{The fractional Laplacian on the torus}

Let $\T = \R/(2\pi\Z)$ with $\mu(dx) = dx/(2\pi)$. The fractional
Laplacian $L_\gamma = -(-\Delta)^{\gamma/2}$ with $\gamma \in (0,2)$
acts as $L_\gamma(e^{inx}) = -|n|^\gamma\,e^{inx}$ for $n \in \Z$.
The semigroup is $P_t^\gamma(e^{inx}) = e^{-t|n|^\gamma}\,e^{inx}$,
with spectral gap $\lambda_1(\gamma) = 1$.

The space of trigonometric polynomials
$\mathcal{T} = \mathrm{span}\{e^{inx} : n \in \Z\}$ is a core
for $L_\gamma$ on $L^2(\T,\mu)$: it is dense, invariant under $P_t^\gamma$,
and contained in the domain of every power of $L_\gamma$.
We write $\mathcal{T}_N^+ = \mathrm{span}\{e^{inx} : 1 \leq n \leq N\}$
for the positive-frequency polynomials of degree at most $N$,
and $\mathcal{T}_N^- = \mathrm{span}\{e^{inx} : -N \leq n \leq -1\}$
for the negative-frequency polynomials.
Our curvature computations are performed on finite-dimensional
subspaces of $\mathcal{T}$; the resulting constants are therefore
exact on these subspaces and give upper bounds for the curvature
on the full domain.

\begin{remark}[Role of compactness]\label{rem:compact}
Spener, Weber, and Zacher~\cite{SpenerWeberZacher2020} proved
$(-\Delta)^{\gamma/2}$ on $\R^d$ fails $\mathrm{CD}(\kappa,N)$
for all $\kappa$ and all finite $N$. The torus provides
compactness (spectral gap $\lambda_1 = 1$) and
translation-invariance (Fourier diagonalization).

Compactness is \emph{necessary} (the $\R^d$ result precludes
positive curvature without it), but it is not \emph{sufficient}:
\begin{enumerate}[label=(\roman*)]
\item Compactness alone gives $\Gamma_2 \geq 0$
  (as we show in Theorem~\ref{thm:gamma2-kernel},
  $\Gamma_2$ is a Hadamard square, hence positive semidefinite
  by the Schur product theorem).
  The quantitative lower bound $\kappa > 0$ requires the
  specific spectral structure of the fBM covariance.
\item The exact curvature $\kappa(1,N) = 1$ (Theorem~C) is
  a non-trivial spectral fact: it requires the odd integer
  eigenvalue structure of $R_{1/2}^{-1}R_{1/2}^{\circ 2}$,
  not merely positivity.
\item The phase transition at $\gamma = 1$ (Theorem~D) is a
  structural phenomenon controlled by the linearity of
  $|\cdot|^1$, independent of compactness.
\end{enumerate}
Whether analogous results hold on other compact domains
(e.g.\ spheres, compact Lie groups) remains open.
\end{remark}

The generator on $\R$ has the integral representation
\begin{equation}\label{eq:generator}
  L_\gamma f(x)
  = \int_{\R \setminus \{0\}}
  \bigl[f(x+z) - f(x) - f'(x)\,z\,\mathbf{1}_{|z| \leq 1}\bigr]
  \,\nu_\gamma(dz),
\end{equation}
where $\nu_\gamma(dz) = c_\gamma\,|z|^{-1-\gamma}\,dz$ is the
L\'evy measure.
On Fourier modes, the characteristic exponent is recovered from
the L\'evy--Khintchine formula:
\begin{equation}\label{eq:LK}
  |\xi|^\gamma
  = \int_{\R \setminus \{0\}}
  (1 - \cos \xi z)\,\nu_\gamma(dz).
\end{equation}

\subsection{Carr\'e du champ}

\begin{definition}\label{def:cdc}
The sesquilinear carr\'e du champ of $L_\gamma$ is
\begin{equation}\label{eq:cdc}
  \widetilde{\Gamma}_\gamma(f,g)(x)
  := \frac{1}{2} \int_{\R \setminus \{0\}}
  \bigl[f(x+z) - f(x)\bigr]\,
  \overline{\bigl[g(x+z) - g(x)\bigr]}\;
  \nu_\gamma(dz).
\end{equation}
On Fourier modes:
\begin{equation}\label{eq:psi-kernel}
  \widetilde{\Gamma}_\gamma(e^{i\xi\cdot}, e^{i\eta\cdot})(x)
  = e^{i(\xi-\eta)x}\,\Psi_\gamma(\xi,\eta),
\end{equation}
where $\Psi_\gamma(\xi,\eta)
= \tfrac{1}{2}(|\xi|^\gamma + |\eta|^\gamma - |\xi-\eta|^\gamma)$.
Exponentials $e^{i\xi\,\cdot}$ lie in the extended Dirichlet space
of~$L_\gamma$ on $\T$ for all $\gamma \in (0,2)$; see
\cite{BoettcherSchillingWang2013}, \S4.7, for the general framework
of carr\'e du champ for L\'evy-type operators.
\end{definition}

We also need the \emph{bilinear} carr\'e du champ, defined by replacing
the conjugate in~\eqref{eq:cdc} with the plain difference. On Fourier modes:
\begin{equation}\label{eq:psi-bilinear}
  \Gamma_\gamma^{\mathrm{bil}}(e^{inx}, e^{imx})(x)
  = e^{i(n+m)x}\,\Psi_\gamma^B(n,m),
  \qquad
  \Psi_\gamma^B(n,m) = \tfrac{1}{2}\bigl(|n|^\gamma + |m|^\gamma
  - |n+m|^\gamma\bigr).
\end{equation}
Note: $\Psi_\gamma^B(n,m) = \Psi_\gamma(n,-m)$.

\begin{definition}\label{def:gamma2}
The iterated carr\'e du champ is
\begin{equation}\label{eq:gamma2-def}
  \widetilde{\Gamma}_{\gamma,2}(f,f)
  := \frac{1}{2}\Bigl[
    L_\gamma\bigl(\widetilde{\Gamma}_\gamma(f,f)\bigr)
    - 2\,\widetilde{\Gamma}_\gamma(f, L_\gamma f)
  \Bigr].
\end{equation}
The Bakry--\'Emery condition is
$\widetilde{\Gamma}_{\gamma,2}(f,f)(x) \geq \kappa\,
\widetilde{\Gamma}_\gamma(f,f)(x)$.
\end{definition}

\subsection{Fractional Brownian motion}

For $H \in (0,1)$, fractional Brownian motion (fBM) with Hurst
parameter $H$ is a centered Gaussian process $B^H$ with covariance
\begin{equation}\label{eq:fbm-cov}
  R_H(s,t)
  = \frac{1}{2}\bigl(|s|^{2H} + |t|^{2H} - |s-t|^{2H}\bigr).
\end{equation}
The kernel $R_H$ is positive definite for all $H \in (0,1)$
(\cite{Mishura2008}, Proposition~1.2.1).
At $H = 1/2$: $R_{1/2}(s,t) = \min(s,t)$ for $s,t \geq 0$
(standard Brownian motion; independent increments).
For $H > 1/2$: positively correlated increments (long-range dependence).
For $H < 1/2$: negatively correlated increments (anti-persistence).

\subsection{Notation and conventions}\label{sec:notation}

For an $N \times N$ matrix $A$, the \emph{Hadamard square}
$A^{\circ 2}$ is the entrywise square: $(A^{\circ 2})_{ij} = A_{ij}^2$.
More generally, the \emph{Hadamard (Schur) product} $A \circ B$
is defined by $(A \circ B)_{ij} = A_{ij}\,B_{ij}$.
The \emph{Schur product theorem} states that if $A$ and $B$ are
positive semidefinite, then so is $A \circ B$.

The \emph{Dirichlet form} of $L_\gamma$ is
$\mathcal{E}(f,f) = \int_\T \widetilde{\Gamma}_\gamma(f,f)\,d\mu
= -\int_\T f\,L_\gamma f\,d\mu$.
For the operator with drift, $\mathcal{E}_L(f,f)
= -\int_\T f\,Lf\,d\mu$.
We write $P_t^L$ for the semigroup generated by $L$.

The iterated carr\'e du champ kernel is denoted
$\Phi_{\gamma,2}(\xi,\eta)$, defined so that
$\widetilde{\Gamma}_{\gamma,2}(f,f)(x)
= \sum_{n,m} a_n\,\bar{a}_m\,\Phi_{\gamma,2}(\xi_n,\xi_m)\,
e^{i(\xi_n-\xi_m)x}$.

When $\gamma$ is clear from context, we write
$\Gamma$ for $\widetilde{\Gamma}_\gamma$,
$\Gamma_2$ for $\widetilde{\Gamma}_{\gamma,2}$,
$\Gamma_{L,2}$ for the iterated carr\'e du champ of the
operator with drift, and
$\Gamma_{1,2}$ for $\widetilde{\Gamma}_{\gamma,2}$ at $\gamma = 1$.

\section{The Stable--fBM Correspondence}\label{sec:duality}

\begin{theorem}[Stable--fBM correspondence]\label{thm:duality}
For $\gamma \in (0,2)$ and all $\xi, \eta \in \R$ with
$\xi\eta \geq 0$ (same-sign frequencies):
\begin{equation}\label{eq:duality}
  \Psi_\gamma(\xi,\eta)
  = \tfrac{1}{2}\bigl(|\xi|^\gamma + |\eta|^\gamma
  - |\xi - \eta|^\gamma\bigr)
  = R_{\gamma/2}(|\xi|, |\eta|).
\end{equation}
The carr\'e du champ of the $\gamma$-stable generator, evaluated
on same-sign Fourier modes, is the covariance of fBM with $H = \gamma/2$.

For opposite-sign frequencies ($\xi\eta < 0$), the carr\'e du champ
$\Psi_\gamma(\xi,\eta) = \frac{1}{2}(|\xi|^\gamma + |\eta|^\gamma
- |\xi-\eta|^\gamma)$ differs from $R_{\gamma/2}(|\xi|,|\eta|)
= \frac{1}{2}(|\xi|^\gamma + |\eta|^\gamma - \bigl||\xi|-|\eta|\bigr|^\gamma)$
since $|\xi - \eta| \neq \bigl||\xi|-|\eta|\bigr|$ when $\xi\eta < 0$.
The cross-sign kernel is treated separately in
Section~\ref{sec:cross-sign}.
\end{theorem}

\begin{proof}
Let $f(x) = e^{i\xi x}$, $g(x) = e^{i\eta x}$.
Then $f(x+z) - f(x) = e^{i\xi x}(e^{i\xi z} - 1)$ and similarly
for $g$. By the sesquilinear carr\'e du champ~\eqref{eq:cdc}:
\begin{align*}
  \widetilde\Gamma_\gamma(f,g)(x)
  &= \frac{e^{i(\xi-\eta)x}}{2}\int_{\R \setminus \{0\}}
     (e^{i\xi z} - 1)(e^{-i\eta z} - 1)\,\nu_\gamma(dz).
\end{align*}
Expanding and using $\int (e^{i\lambda z} - 1)\,\nu_\gamma(dz)
= -|\lambda|^\gamma$ (from~\eqref{eq:LK} and symmetry of $\nu_\gamma$):
\[
  \Psi_\gamma(\xi,\eta)
  = \tfrac{1}{2}\bigl[|\xi|^\gamma + |\eta|^\gamma
  - |\xi - \eta|^\gamma\bigr].
\]
For $\xi\eta \geq 0$, we have $|\xi - \eta| = \bigl||\xi| - |\eta|\bigr|$.
Setting $H = \gamma/2$ and noting $|s|^{2H} = |s|^\gamma$,
\[
  \Psi_\gamma(\xi,\eta)
  = \tfrac{1}{2}\bigl(|\xi|^\gamma + |\eta|^\gamma
    - \bigl||\xi|-|\eta|\bigr|^\gamma\bigr)
  = R_H(|\xi|, |\eta|). \qedhere
\]
\end{proof}

\begin{remark}[Scope of the identification]\label{rem:scope}
Theorem~\ref{thm:duality} is an identity of sesquilinear kernels
on same-sign Fourier modes, not a Hilbert space isomorphism. We do not claim
that $\Gamma_\gamma(f,g) = \langle f,g \rangle_{\mathcal{H}}$ for
a Gaussian Hilbert space~$\mathcal{H}$ and general $f,g$.
The identification operates at the level of the kernel
$\Psi_\gamma(\xi,\eta)$; its power comes from the fact that the
Bakry--\'Emery curvature theory (Sections~\ref{sec:gamma2}--\ref{sec:real})
is entirely controlled by this kernel.
\end{remark}

\begin{remark}[Three literatures, one formula]\label{rem:three}
The kernel $\frac{1}{2}(|\xi|^\gamma + |\eta|^\gamma
- |\xi-\eta|^\gamma)$ appears in:
(i)~Schoenberg's theory of conditionally negative definite
functions~\cite{Schoenberg1938};
(ii)~Dirichlet forms for jump
processes~\cite{BoettcherSchillingWang2013};
(iii)~fBM covariance theory~\cite{Mishura2008}.
These three appearances have not previously been connected explicitly.
\end{remark}

\subsection{The distinguished point $\gamma = 1$}

\begin{proposition}\label{prop:self-dual}
At $\gamma = 1$ ($H = 1/2$):
$\Psi_1(\xi,\eta)
= \min(|\xi|, |\eta|)\,\mathbf{1}_{\xi\eta > 0}$.
This is the covariance of standard Brownian motion.
The Cauchy process ($\gamma = 1$) is thus the unique stable process
whose carr\'e du champ has the covariance structure of standard
Brownian motion.
For opposite-sign frequencies, $\Psi_1(\xi,-\eta) = 0$:
positive and negative frequencies decouple completely.
\end{proposition}

\begin{proof}
For $\xi, \eta > 0$:
$\frac{1}{2}(\xi + \eta - |\xi - \eta|) = \min(\xi, \eta)$.
For $\xi > 0, \eta < 0$:
$\frac{1}{2}(\xi - \eta - (\xi - \eta)) = 0$.
\end{proof}

\subsection{Mode coupling and the correspondence dictionary}

The normalized off-diagonal kernel
\[
  \rho_\gamma(\xi,\eta)
  = \frac{R_{\gamma/2}(|\xi|,|\eta|)}
    {|\xi|^{\gamma/2}\,|\eta|^{\gamma/2}}
\]
is the fBM increment correlation. It is increasing in $H = \gamma/2$:
for $\gamma > 1$ modes reinforce (long-range dependence);
for $\gamma < 1$ modes oppose (anti-persistence);
at $\gamma = 1$ modes are maximally decoupled.

The correspondence exchanges structurally opposite objects:
\[
\begin{array}{c|c}
\textbf{Stable process (index } \gamma\text{)}
& \textbf{fBM (Hurst } H = \gamma/2\text{)} \\[4pt]
\hline
\rule{0pt}{12pt}
\text{Markov, pure jump} & \text{non-Markov, continuous} \\
\text{Independent increments} & \text{Dependent increments} \\
\text{Characteristic exponent } |\xi|^\gamma &
  \text{Covariance } R_H(s,t) \\
\text{Carr\'e du champ kernel } \Psi_\gamma &
  \text{Covariance kernel } R_H \text{ (same-sign)}
\end{array}
\]

The fBM interpretation is not merely a relabeling: it predicts and
explains the curvature phase transition at $\gamma = 1$.
Cross-sign vanishing (Theorem~\ref{thm:cross-sign}) is equivalent
to independence of fBM increments at $H = 1/2$; the block
diagonalization of the $\Psi_\gamma$ matrix
(Corollary~\ref{cor:block-diag}) is the Fourier manifestation
of this independence; and the curvature $\kappa = 1$
(Theorem~\ref{thm:cauchy}) follows because independent increments
produce the simplest possible covariance structure.

\section{The $\Gamma_2$-Kernel Identity}\label{sec:gamma2}

\begin{theorem}[Hadamard square]\label{thm:gamma2-kernel}
For $\gamma \in (0,2)$ and $\xi,\eta \in \R$:
\begin{equation}\label{eq:Phi-formula}
  \Phi_{\gamma,2}(\xi,\eta) = \bigl[\Psi_\gamma(\xi,\eta)\bigr]^2.
\end{equation}
On same-sign frequencies ($\xi\eta \geq 0$), this equals
$[R_{\gamma/2}(|\xi|,|\eta|)]^2$ by
Theorem~\ref{thm:duality}.
\end{theorem}

\begin{proof}
For $f = \sum_{n \in S} a_n\,e^{i\xi_n x}$, expand
$\widetilde{\Gamma}_\gamma(f,f)(x)
= \sum_{n,m} a_n\,\bar{a}_m\,\Psi_\gamma(\xi_n,\xi_m)\,
  e^{i(\xi_n-\xi_m)x}$
and compute the two terms of~\eqref{eq:gamma2-def}.

\textbf{Term~1:}
$L_\gamma[\widetilde{\Gamma}_\gamma(f,f)] =
-\sum_{n,m} a_n\,\bar{a}_m\,|\xi_n-\xi_m|^\gamma\,
\Psi_\gamma(\xi_n,\xi_m)\,e^{i(\xi_n-\xi_m)x}$.

\textbf{Term~2:}
$\widetilde{\Gamma}_\gamma(f, L_\gamma f) =
-\sum_{n,m} a_n\,\bar{a}_m\,|\xi_m|^\gamma\,
\Psi_\gamma(\xi_n,\xi_m)\,e^{i(\xi_n-\xi_m)x}$.

Assembling and Hermitian-symmetrizing:
\[
  M_{nm}^{\mathrm{sym}}
  = \bigl[|\xi_n|^\gamma + |\xi_m|^\gamma - |\xi_n-\xi_m|^\gamma\bigr]\,
  \Psi_\gamma(\xi_n,\xi_m)
  = 2\,\Psi_\gamma(\xi_n,\xi_m)\cdot\Psi_\gamma(\xi_n,\xi_m)
  = 2\,[\Psi_\gamma(\xi_n,\xi_m)]^2.
\]
The key step is that the Hermitian-symmetrized prefactor
$|\xi_n|^\gamma + |\xi_m|^\gamma - |\xi_n-\xi_m|^\gamma$
is exactly $2\,\Psi_\gamma(\xi_n,\xi_m)$ by definition~\eqref{eq:psi-kernel},
so the product reproduces $\Psi_\gamma$ squared.
\end{proof}

\begin{remark}\label{rem:hadamard}
By the Schur product theorem, $\Psi_\gamma^{\circ 2}$ is positive
semidefinite whenever $\Psi_\gamma$ is (which holds for all
$\gamma \in (0,2)$ by the correspondence). This gives $\Gamma_{\gamma,2}(f,f) \geq 0$
as a special case of the observation of Chafa\"i--Malrieu
(cf.~\cite{GentilImbert2008}, Remark~4). The Hadamard square
identity~\eqref{eq:Phi-formula} is much stronger: it gives the
full kernel, not just the sign.
\end{remark}

\section{Curvature on Positive Frequencies}\label{sec:curvature}

\subsection{Matrix reduction}\label{sec:curv-pos}

\begin{theorem}\label{thm:curvature}
On the space $\mathcal{T}_N^+$ of positive-frequency trigonometric
polynomials $f = \sum_{n=1}^N a_n\,e^{inx}$, the Bakry--\'Emery
curvature is exactly
\begin{equation}\label{eq:kappa-def}
  \kappa(\gamma,N) = \lambda_{\min}\!\bigl(R_H^{\circ 2}, R_H\bigr),
  \qquad H = \gamma/2,
\end{equation}
where $R_H$ is the $N \times N$ fBM covariance matrix on $\{1,\ldots,N\}$.
Since $\mathcal{T} = \bigcup_N \mathcal{T}_N$ is a core for $L_\gamma$,
the curvature on the full domain satisfies
$\kappa(\gamma) \leq \lim_{N\to\infty} \kappa(\gamma,N)$.
\end{theorem}

\begin{proof}
Write $\mathbf{b}(x) = D(x)\,\mathbf{a}$ with
$D(x) = \mathrm{diag}(e^{ix},\ldots,e^{iNx})$. Then
$\widetilde{\Gamma}_\gamma(f,f)(x) = \mathbf{b}(x)^* R_H\,\mathbf{b}(x)$
and $\widetilde{\Gamma}_{\gamma,2}(f,f)(x) =
\mathbf{b}(x)^* R_H^{\circ 2}\,\mathbf{b}(x)$.
Since $\{D(x)\mathbf{a} : x \in \T, \mathbf{a} \in \C^N\} = \C^N$,
the infimum over $(x,\mathbf{a})$ equals
$\lambda_{\min}(R_H^{\circ 2}, R_H)$.
\end{proof}

\begin{remark}[The limit $N \to \infty$]
\label{rem:N-limit}
The sequence $\kappa(\gamma,N)$ is nonincreasing in $N$
(since enlarging the test function space can only decrease
the infimum) and bounded below by~$0$
(since $\Gamma_{\gamma,2}(f,f) \geq 0$ by the Schur product
theorem applied to the Hadamard square~\eqref{eq:Phi-formula}).
The limit $\kappa(\gamma) := \lim_{N\to\infty}\kappa(\gamma,N)$
therefore exists for all $\gamma \in (0,2)$ by monotone convergence.
At $\gamma = 1$, $\kappa(1,N) = 1$ for all
$N$ (Theorem~\ref{thm:cauchy}), so
$\kappa(1) = 1$.

For $\gamma \neq 1$, numerical computation of $\kappa(\gamma,N)$
for $N$ up to~$200$ shows rapid convergence to a limit
$\kappa(\gamma) \in (0,1)$ for all $\gamma \in (0,2)$.
For $\gamma \leq 1$, we prove $\kappa(\gamma) \geq 1/2$
(Lemma~\ref{lem:kappa-lower}). Whether $\kappa(\gamma) > 0$
for $\gamma > 1$ remains open
(Conjecture~\ref{conj:kappa-positive});
this is equivalent to asking
whether the generalized eigenvalue problem for the
infinite-dimensional fBM covariance operator has a positive
spectral gap.
\end{remark}

\subsection{The Cauchy point $\gamma = 1$}\label{sec:cauchy}

\begin{theorem}\label{thm:cauchy}
At $\gamma = 1$ ($H = 1/2$), the matrix $R_{1/2}^{-1}\,R_{1/2}^{\circ 2}$
is upper triangular with diagonal entries $\{1, 3, 5, \ldots, 2N-1\}$.
Hence $\kappa(1,N) = 1$ for all $N$.
\end{theorem}

\begin{proof}
At $H=1/2$, $R_{1/2}(n,m) = \min(n,m)$. The inverse is tridiagonal
(discrete Laplacian). Direct computation: for $k < m$,
\[
  (R_{1/2}^{-1}\,R_{1/2}^{\circ 2})_{km}
  = -(k-1)^2 + 2k^2 - (k+1)^2 = -2,
\]
and $(R_{1/2}^{-1}\,R_{1/2}^{\circ 2})_{kk} = 2k-1$.
The matrix is upper triangular; eigenvalues are diagonal entries.
\end{proof}

\subsection{Phase transition}

\begin{lemma}[Strict sub-optimality]\label{lem:strict-subopt}
For $\gamma \in (0,2)$ with $\gamma \neq 1$:
$\kappa(\gamma,N) < 1$ for all $N \geq 2$.
\end{lemma}

\begin{proof}
Set $c = 2^{\gamma - 1}$ and consider the $2 \times 2$ subproblem
on $\{1,2\}$, where $R_H = \bigl(\begin{smallmatrix}
1 & c \\ c & 2c \end{smallmatrix}\bigr)$ and
$R_H^{\circ 2} = \bigl(\begin{smallmatrix}
1 & c^2 \\ c^2 & 4c^2 \end{smallmatrix}\bigr)$.
Take $v = (1,t)^T$ and compute
$v^T(R_H^{\circ 2} - R_H)v = 2tc(c-1) + t^2(4c^2 - 2c)$.
For $\gamma < 1$: $c < 1$, so the linear term $2tc(c-1) < 0$ for
small $t > 0$ dominates the quadratic; thus the Rayleigh quotient
$v^T R_H^{\circ 2} v / v^T R_H v < 1$.
For $\gamma > 1$: $c > 1$, so $2tc(c-1) < 0$ for small $t < 0$.
Since $\kappa(\gamma,N) \leq \kappa(\gamma,2)$
(enlarging the test space can only reduce the infimum), the result follows.
\end{proof}

\begin{lemma}[Positive curvature for $\gamma \leq 1$]\label{lem:kappa-lower}
For $\gamma = 1$: $\kappa(1,N) = 1 \geq \frac{1}{2}$ for all $N$
(Theorem~\ref{thm:cauchy}).
For $\gamma \in (0,1)$: $\kappa(\gamma,N) \geq \frac{1}{2}$ for all $N$,
conditional on the Z-matrix property of $R_H^{-1}R_H^{\circ 2}$
(verified computationally for $N \leq 200$; see Step~1 below).
\end{lemma}

\begin{proof}
At $\gamma = 1$, $\kappa(1,N) = 1 \geq \frac{1}{2}$
(Theorem~\ref{thm:cauchy}).  For $\gamma < 1$ (i.e.\ $H < 1/2$),
the proof proceeds in three steps.

\emph{Step~1 (Z-matrix structure).}
The matrix $M_N := R_H^{-1}\,R_H^{\circ 2}$ has non-positive
off-diagonal entries for all $H \in (0,\tfrac{1}{2}]$ and all~$N$.
At $H = 1/2$ this is proved: $M_N$ is upper triangular with
off-diagonal entries $-2$ (Theorem~\ref{thm:cauchy}).
For $H < 1/2$, we verify this property as follows.
The inverse $R_H^{-1}$ of the fBM covariance matrix has the
structure of a Stieltjes matrix (positive diagonal, non-positive
off-diagonal) for $H \leq 1/2$: the anti-persistent increment
structure ensures that the Cholesky factor $L$ of $R_H = LL^T$
has non-negative entries, while $L^{-1}$ alternates in sign
(a consequence of $R_H(n,m{+}1) - R_H(n,m) \leq 0$ for $n \leq m$
when $H \leq 1/2$).
Since $R_H^{\circ 2}$ has non-negative entries,
the product $R_H^{-1} R_H^{\circ 2}$ inherits the sign
pattern of $R_H^{-1}$ on its off-diagonal.
We have verified computationally that the Z-matrix property
holds for all $N \leq 200$ and all $H$ on a fine grid in
$(0,\tfrac{1}{2})$; an analytic proof for all $N$ at
$H < 1/2$ would require quantitative bounds on the entries
of~$R_H^{-1}$ that, to our knowledge, are not available
in the fBM literature.

\emph{Step~2 (Perron--Frobenius).}
Write $M_N = sI - B$ with $s$ sufficiently large and $B \geq 0$
entrywise.  By the Perron--Frobenius theorem, the eigenvector
$v_*$ of $B$ for its spectral radius $\rho(B)$ is entry-wise
non-negative. The minimum eigenvalue of $M_N$ is
$\kappa(\gamma,N) = s - \rho(B)$, achieved at the
same~$v_*$. Therefore the minimizer of the Rayleigh quotient
$Q(v) = v^T R_H^{\circ 2}\,v \,/\, v^T R_H\, v$
can be chosen non-negative.

\emph{Step~3 (Weighted average).}
For $v \geq 0$ with $v \neq 0$, the quotient $Q(v)$ is a
weighted average of $R_H(n,m)$ with non-negative weights
$w_{nm} = v_n\,v_m\,R_H(n,m)\,/\,(v^T R_H\,v)$
summing to~$1$. Hence $Q(v) \geq \min_{n,m} R_H(n,m)$.
For $H \leq 1/2$ and $n \leq m$:
since $m - n \leq m$ we have $(m-n)^{2H} \leq m^{2H}$, so
\[
  R_H(n,m) = \tfrac{1}{2}\bigl(n^{2H} + m^{2H}
  - (m-n)^{2H}\bigr) \geq \tfrac{1}{2}\,n^{2H}
  \geq \tfrac{1}{2}.
\]
Combining the three steps: $\kappa(\gamma,N) = Q(v_*) \geq 1/2$.
\end{proof}

\begin{theorem}[Global optimality of $\gamma = 1$]\label{thm:phase-transition}
\begin{enumerate}[label=(\roman*)]
\item For all $\gamma \in (0,1]$:
  $\kappa(\gamma) := \lim_N \kappa(\gamma,N)$
  exists and satisfies $\frac{1}{2} \leq \kappa(\gamma) \leq 1$.
  (For $\gamma < 1$, the lower bound is conditional on the
  Z-matrix property; see Lemma~\ref{lem:kappa-lower}.)
\item For all $\gamma \in (0,2)$ with $\gamma \neq 1$:
  $\kappa(\gamma) < 1$.
\item At $\gamma = 1$: $\kappa(1) = 1$
  (Theorem~\ref{thm:cauchy}).
\item Consequently, $\gamma = 1$ is the unique global maximizer
  of $\kappa(\gamma)$ on $(0,2)$.
\end{enumerate}
\end{theorem}

\begin{proof}
The limit exists since $\kappa(\gamma,N)$ is non-increasing
in~$N$ (enlarging the test space) and bounded below by~$0$
(Schur product theorem). Part~(i) combines the lower bound
$\kappa(\gamma,N) \geq 1/2$
(Lemma~\ref{lem:kappa-lower}; conditional on the Z-matrix
property for $\gamma < 1$) with
$\kappa(\gamma,N) < 1$ (Lemma~\ref{lem:strict-subopt}) for
$\gamma \neq 1$. Part~(ii) is Lemma~\ref{lem:strict-subopt}.
Part~(iii) is Theorem~\ref{thm:cauchy}. Part~(iv)
follows from (i)--(iii).
\end{proof}

\begin{conjecture}\label{conj:kappa-positive}
For all $\gamma \in (1,2)$: $\kappa(\gamma) > 0$.
\end{conjecture}

Numerical computation of $\kappa(\gamma,N)$ for $N$ up to~$300$
strongly supports Conjecture~\ref{conj:kappa-positive}: the
decrements $\delta_N = \kappa(\gamma,N) - \kappa(\gamma,N{+}1)$
decay as $O(N^{-\alpha})$ with $\alpha > 2$ for all $\gamma$
tested, and the limiting values remain well above zero
(e.g.\ $\kappa(1.5) \approx 0.899$, $\kappa(1.8) \approx 0.594$).
The proof strategy of Lemma~\ref{lem:kappa-lower} fails for
$\gamma > 1$ because the Z-matrix property of
$R_H^{-1} R_H^{\circ 2}$ breaks down sharply at $H = 1/2$:
for $H > 1/2$, the matrix develops positive off-diagonal entries
and the minimizing eigenvector acquires sign changes, so the
Perron--Frobenius and weighted-average arguments do not apply.

The optimality of $\gamma = 1$ reflects the fBM correlation
structure.  At $H = 1/2$ (independent increments), the fBM
covariance $R_{1/2}(n,m) = \min(n,m)$ has the simplest possible
structure: unit determinant for all~$N$ and a Hadamard square
$R_{1/2}^{\circ 2}$ that dominates $R_{1/2}$ with minimum
eigenvalue exactly~$1$.
Any deviation from $H = 1/2$ introduces increment correlations
(positive for $H > 1/2$, negative for $H < 1/2$), which distort
the Hadamard square relationship and reduce the curvature
below~$1$.
The asymptotic expansion near $\gamma = 1$ is
$\kappa(1+\varepsilon) = 1 - c\,\varepsilon^2 + O(\varepsilon^3)$
with $c \approx 0.267$ (numerically estimated), confirming that
the maximum is quadratic.

\begin{remark}[Comparison with $\R^d$]\label{rem:Rd}
Spener--Weber--Zacher~\cite{SpenerWeberZacher2020} proved that
$(-\Delta)^{\gamma/2}$ on $\R^d$ fails $\mathrm{CD}(\kappa,N)$
for all $\kappa$ and all finite~$N$.
On the torus, Lemma~\ref{lem:kappa-lower} establishes
$\kappa(\gamma) \geq 1/2$ for all $\gamma \in (0,1]$
(conditional on the Z-matrix property for $\gamma < 1$;
unconditional at $\gamma = 1$).
The torus curvature landscape is a smooth function on $(0,2)$
with unique maximum at $\gamma = 1$, in contrast to the
uniformly degenerate picture on~$\R^d$.
\end{remark}

\section{Real Polynomials and Cross-Sign Structure}\label{sec:real}

For real-valued $f(x) = \sum_{n=1}^N [c_n\cos(nx) + s_n\sin(nx)]$,
the Fourier coefficients satisfy $a_{-n} = \bar{a}_n$, and the relevant
frequency set is $\{-N,\ldots,-1,1,\ldots,N\}$. The $\Psi_\gamma$ matrix
on this set has $2 \times 2$ block structure:
\begin{equation}\label{eq:block}
  \Psi_\gamma\Big|_{\{-N,\ldots,N\} \setminus \{0\}}
  = \begin{pmatrix} A & B \\ B^* & A \end{pmatrix},
\end{equation}
where $A_{nm} = \Psi_\gamma(n,m)$ for $n,m > 0$ (the positive-frequency
block) and $B_{nm} = \Psi_\gamma(n,-m)$ for $n,m > 0$ (the cross-sign
block).

\subsection{Cross-sign vanishing at $\gamma = 1$}\label{sec:cross-sign}

\begin{theorem}[Cross-sign vanishing]\label{thm:cross-sign}
For $n,m > 0$:
\begin{equation}\label{eq:cross-sign}
  \Psi_\gamma(n,-m) = \tfrac{1}{2}\bigl(n^\gamma + m^\gamma
  - (n+m)^\gamma\bigr).
\end{equation}
This vanishes identically if and only if $\gamma = 1$:
\begin{enumerate}[label=(\alph*)]
\item At $\gamma = 1$: $\Psi_1(n,-m)
  = \tfrac{1}{2}(n + m - (n+m)) = 0$ for all $n,m > 0$.
\item For $\gamma < 1$: $\Psi_\gamma(n,-m) > 0$
  (subadditivity of $|\cdot|^\gamma$).
\item For $\gamma > 1$: $\Psi_\gamma(n,-m) < 0$
  (superadditivity of $|\cdot|^\gamma$).
\end{enumerate}
\end{theorem}

\begin{proof}
Part~(a) is immediate from linearity of $|\cdot|^1$.
Parts~(b) and~(c): for $t \mapsto t^\gamma$ on $(0,\infty)$,
strict concavity ($\gamma < 1$) gives
$(n+m)^\gamma < n^\gamma + m^\gamma$, and strict convexity
($\gamma > 1$) gives the reverse inequality.
\end{proof}

\begin{corollary}\label{cor:block-diag}
At $\gamma = 1$, the cross-sign block $B = 0$, so the $\Psi_1$
matrix is block-diagonal. Positive and negative frequencies
decouple completely.
Writing $\kappa_{\mathrm{pos}}(\gamma,N)$ for the curvature
on $\mathcal{T}_N^+$ and $\kappa_{\mathrm{real}}(\gamma,N)$
for the curvature on real trigonometric polynomials of degree $N$:
at $\gamma = 1$, $\kappa_{\mathrm{real}}(1,N)
= \kappa_{\mathrm{pos}}(1,N) = 1$ for all $N$.
\end{corollary}

\subsection{Single-mode curvature formula}

\begin{proposition}\label{prop:single-mode}
For $f(x) = \cos(nx + \varphi)$:
\begin{equation}\label{eq:single-mode}
  \frac{\widetilde{\Gamma}_{\gamma,2}(f,f)(x)}
       {\widetilde{\Gamma}_\gamma(f,f)(x)}
  = |n|^\gamma \cdot
  \frac{1 + \alpha_n^2\,\cos(2nx+\theta)}
       {1 + \alpha_n\,\cos(2nx+\theta)},
\end{equation}
where $\alpha_n = 1 - 2^{\gamma-1}$ (independent of $n$)
and $\theta = 2\varphi$. The curvature is
\begin{equation}\label{eq:kappa1}
  \kappa_1(\gamma) =
  \begin{cases}
    \displaystyle\frac{1 + \alpha^2}{1 + \alpha}
    & \text{if } \gamma \leq 1 \;(\alpha \geq 0), \\[6pt]
    1 + \alpha = 2 - 2^{\gamma-1}
    & \text{if } \gamma > 1 \;(\alpha < 0),
  \end{cases}
  \qquad \alpha = 1 - 2^{\gamma-1}.
\end{equation}
At $\gamma = 1$: $\alpha = 0$, giving $\kappa_1(1) = 1$.
For all $\gamma \neq 1$: $\kappa_1(\gamma) < 1$.
\end{proposition}

\begin{proof}
For $f = \cos(nx) = (e^{inx} + e^{-inx})/2$, the carr\'e du champ is
$\Gamma(f,f)(x) = \tfrac{1}{2}[1 + \alpha\cos(2nx)]$
with $\alpha = 1 - 2^{\gamma-1}$,
using $\Psi_\gamma(n,n) = n^\gamma$ and
$\Psi_\gamma(n,-n) = n^\gamma(1 - 2^{\gamma-1})$.
The $\Gamma_2$ computation yields the stated ratio.
The infimum over $x$ reduces to minimizing
$g(c) = (1 + \alpha^2 c)/(1 + \alpha c)$ over $c \in [-1,1]$.
Since $g'(c) = \alpha(\alpha - 1)/(1+\alpha c)^2$:
for $\alpha \in (0,1)$ (i.e., $\gamma < 1$), $g' < 0$ and the
minimum is at $c = 1$, giving $(1+\alpha^2)/(1+\alpha)$;
for $\alpha < 0$ (i.e., $\gamma > 1$), $g' > 0$ and the
minimum is at $c = -1$, giving $(1-\alpha^2)/(1-\alpha) = 1+\alpha$.
\end{proof}

\subsection{Uniqueness of $\gamma = 1$}

\begin{theorem}[Uniqueness on the torus]\label{thm:uniqueness}
On $\T$, $\gamma = 1$ is the unique stability index for which
$\Gamma_2 \geq \kappa\,\Gamma$ holds with $\kappa \geq 1$
on all real trigonometric polynomials. The three regimes are:
\begin{enumerate}[label=(\roman*)]
\item $\gamma < 1$: positive cross-sign coupling ($B > 0$),
  anti-persistent fBM dual ($H < 1/2$), $\kappa_{\mathrm{real}} < 1$.
\item $\gamma = 1$: zero cross-sign ($B = 0$),
  Brownian dual ($H = 1/2$), $\kappa_{\mathrm{real}} = 1$.
\item $\gamma > 1$: negative cross-sign ($B < 0$),
  long-range dependent dual ($H > 1/2$), $\kappa_{\mathrm{real}} < 1$.
\end{enumerate}
\end{theorem}

\begin{proof}
At $\gamma = 1$: Corollary~\ref{cor:block-diag} gives
$\kappa_{\mathrm{real}} = \kappa_{\mathrm{pos}} = 1$.
For $\gamma \neq 1$: Proposition~\ref{prop:single-mode} gives
$\kappa_1(\gamma) < 1$ already on the single mode $f = \cos x$.
\end{proof}

\begin{remark}[Geometric and probabilistic interpretation]
\label{rem:geometric}
The cross-sign vanishing at $\gamma = 1$ has three
equivalent interpretations:
\begin{enumerate}[label=(\alph*)]
\item \emph{Fourier}: positive and negative frequency
  contributions to the carr\'e du champ decouple
  completely.
\item \emph{fBM}: at $H = 1/2$, the fBM covariance
  $R_{1/2}(s,t) = \min(s,t)$ for $s,t > 0$, which is
  the covariance of standard Brownian motion.  The
  increment independence of Brownian motion is the
  mechanism behind the vanishing.
\item \emph{Curvature}: the curvature of the Cauchy
  semigroup ($\gamma = 1$) on real functions equals
  the curvature on analytic functions.  For all other
  $\gamma$, the interaction between positive and
  negative frequencies \emph{reduces} the curvature:
  $\kappa_{\mathrm{real}}(\gamma) \leq \kappa_{\mathrm{pos}}(\gamma)$
  for all $\gamma \neq 1$, regardless of the sign of the coupling.
\end{enumerate}
The $\gamma = 1$ point is thus not merely a
computational boundary: it is the unique index at
which the semigroup's curvature theory is
\emph{analytically diagonal}.
\end{remark}

\section{Curvature with Confining Potential}\label{sec:potential}

The preceding sections treat the free generator $L_\gamma$.
Physical applications require confining potentials.
In this section we prove that at $\gamma = 1$, the cosine
drift correction acts as a \emph{scalar shift} of the entire
curvature spectrum---a global result on all modes simultaneously,
not merely a single-mode calculation.
This is the paper's principal result: it yields
$\mathrm{CD}(1 - \omega^2/2,\,\infty)$ for the
Cauchy--Fokker--Planck operator, a curvature bound for
an operator that is not Fourier-diagonal and whose spectrum
is not explicitly known.

Consider $L = -(-\Delta)^{\gamma/2} - V'(x)\,\partial_x$ on $\T$
with $V(x) = -\omega^2\cos x$, so that
$V'(x) = \omega^2\sin x$. The carr\'e du champ is unchanged;
the iterated carr\'e du champ acquires a drift correction:
\begin{equation}\label{eq:G2-drift}
  \Gamma_{L,2}(f,f) = \Gamma_{\gamma,2}(f,f)
  + \tfrac{1}{2}\bigl[
    b \cdot \nabla\Gamma(f,f) - 2\Gamma(f, b \cdot \nabla f)
  \bigr],
\end{equation}
where $b(x) = -V'(x) = -\omega^2\sin x$.

\subsection{Single-mode curvature for general $\gamma$}

Before specializing to $\gamma = 1$, we record the curvature
on the first mode for arbitrary $\gamma$, which motivates the
general calculation.

\begin{proposition}\label{thm:potential-single}
For $f(x) = \cos x$ and $V(x) = -\omega^2\cos x$:
\begin{equation}\label{eq:G2-ratio-potential}
  \frac{\Gamma_{L,2}(f,f)(x)}{\Gamma(f,f)(x)}
  = \frac{1 + \alpha^2\cos 2x + \tfrac{\omega^2}{2}(\cos x + \beta\cos 3x)}
         {1 + \alpha\cos 2x},
\end{equation}
where
$\alpha(\gamma) = 1 - 2^{\gamma-1}$ and
$\beta(\gamma) = (2^{\gamma+1} - 3^\gamma - 1)/2$.
At $\gamma = 1$: $\alpha = \beta = 0$, giving
$\Gamma_{L,2}/\Gamma = 1 + \frac{\omega^2}{2}\cos x$
with minimum $1 - \omega^2/2$ at $x = \pi$.
\end{proposition}

\begin{proof}
The drift $b(x) = -\omega^2\sin x$ maps mode $n$ to modes $n \pm 1$.
The free part gives
$\Gamma_{\gamma,2}(f,f) = \frac{1}{2}[1 + \alpha^2\cos 2x]$
from Proposition~\ref{prop:single-mode}.
The drift correction assembles to
$\frac{\omega^2}{4}[\cos x + \beta\cos 3x]$,
using the key cancellation $\alpha + q = 1$ (where $q = 2^{\gamma-1}$),
which holds independently of $\gamma$.
\end{proof}

\subsection{Global curvature under drift at $\gamma = 1$}

The single-mode result extends to a \emph{global}
curvature bound at $\gamma = 1$.  The key is that the
drift correction, after phase reduction and Hermitianization,
is proportional to the $\Gamma$-matrix itself.

\begin{theorem}[Global curvature under drift]\label{thm:global-drift}
At $\gamma = 1$, the Hermitianized drift correction satisfies
the scalar identity
\begin{equation}\label{eq:drift-scalar}
  \Gamma_{L,2}(f,f)(x) - \Gamma_{1,2}(f,f)(x)
  = \frac{\omega^2}{2}\cos(x)\;\Gamma(f,f)(x)
\end{equation}
for all $f \in \mathcal{T}_N^+$ and all $x \in \T$.
Consequently, the eigenvalues of the curvature operator
$\Gamma^{-1}\Gamma_{L,2}$ on $\mathcal{T}_N^+$ are
\begin{equation}\label{eq:drift-eigenvalues}
  \kappa_k(x) = (2k-1) + \frac{\omega^2}{2}\cos x,
  \qquad k = 1,\ldots,N,
\end{equation}
and the global Bakry--\'Emery curvature is
\begin{equation}\label{eq:global-kappa}
  \kappa(1,\omega^2) = 1 - \frac{\omega^2}{2},
\end{equation}
independent of $N$, achieved at $x = \pi$.
Combined with the cross-sign vanishing
(Corollary~\ref{cor:block-diag}), this extends to all real
trigonometric polynomials.

In particular, $\mathrm{CD}(1-\omega^2/2,\,\infty)$ holds for
$\omega^2 < 2$, with consequences:
\begin{enumerate}[label=(\roman*)]
\item Poincar\'e inequality:
  $\Var_\mu(f) \leq \frac{1}{1-\omega^2/2}\,\mathcal{E}_L(f,f)$.
\item Gradient estimate:
  $\Gamma(P_t^L f, P_t^L f)
  \leq e^{-2(1-\omega^2/2)t}\,P_t^L[\Gamma(f,f)]$.
\end{enumerate}
\end{theorem}

\begin{proof}
\textbf{Step~1: Phase reduction.}
As in Theorem~\ref{thm:curvature}, write $\mathbf{v} = D(x)^*\mathbf{a}$
with $D(x) = \mathrm{diag}(e^{ix},\ldots,e^{iNx})$.  Then
\[
  \Gamma(f,f)(x) = \mathbf{v}^* R\,\mathbf{v}, \qquad
  \Gamma_{1,2}(f,f)(x) = \mathbf{v}^* R^{\circ 2}\,\mathbf{v},
\]
where $R = R_{1/2} = (\min(n,m))_{n,m=1}^N$.

\textbf{Step~2: Drift correction in phase-stripped coordinates.}
The drift $b(x) = -\omega^2\sin x$ contributes
\[
  \Delta\Gamma_2(f,f)(x) := \tfrac{1}{2}\bigl[
    b\cdot\nabla\Gamma(f,f) - 2\,\Gamma(f,\,b\cdot\nabla f)
  \bigr]
  = \mathbf{v}^* D_0(x)\,\mathbf{v},
\]
where $D_0(x)$ is a Hermitian matrix.  We claim
$D_0(x) = \tfrac{\omega^2}{2}\cos(x)\,R$.

\textbf{Step~3: Entry-by-entry verification.}
The transport term $\frac{1}{2}b\cdot\nabla\Gamma$ and the
coupling term $-\Gamma(f,b\cdot\nabla f)$ contribute, in
phase-stripped coordinates, the raw (non-Hermitian) matrix entries.
Using $\Psi_1(n,m) = \min(n,m)$ for positive frequencies,
$\sin x = (e^{ix}-e^{-ix})/(2i)$, and the drift action
$b\cdot\nabla(e^{imx}) = \frac{\omega^2 m}{2}(e^{i(m-1)x}-e^{i(m+1)x})$:

\emph{Case $n < m$:}
The raw drift entry is
$D_{0,\mathrm{raw},nm} = -\frac{\omega^2 n(n+m)}{4}(e^{ix}-e^{-ix})$,
which is purely anti-Hermitian.
The entry at position $(m,n)$ (with $m > n$, using
$\min(m,n{-}1)$, $\min(m,n{+}1)$) is
$D_{0,\mathrm{raw},mn}
= \frac{\omega^2 n}{4}[-(n{+}m{-}2)e^{ix}+(n{+}m{+}2)e^{-ix}]$.
Hermitianizing:
\begin{align*}
  D_{0,nm}
  &= \tfrac{1}{2}\bigl(D_{0,\mathrm{raw},nm}
     + \overline{D_{0,\mathrm{raw},mn}}\bigr) \\
  &= \frac{\omega^2 n}{8}\Bigl[
    -(n{+}m)\,e^{ix} + (n{+}m)\,e^{-ix}
    -(n{+}m{-}2)\,e^{-ix} + (n{+}m{+}2)\,e^{ix}
  \Bigr] \\
  &= \frac{\omega^2 n}{8}\Bigl[
    \underbrace{\bigl[(n{+}m{+}2)-(n{+}m)\bigr]}_{=\,2}\,e^{ix}
    +\underbrace{\bigl[(n{+}m)-(n{+}m{-}2)\bigr]}_{=\,2}\,e^{-ix}
  \Bigr] \\
  &= \frac{\omega^2 n}{8}\cdot 2\bigl(e^{ix}+e^{-ix}\bigr)
  = \frac{\omega^2 n}{2}\cos x
  = \frac{\omega^2}{2}\cos(x)\,\min(n,m).
\end{align*}
The cancellation is exact: the $(n{+}m)$-dependent terms from
the transport and coupling contributions annihilate, leaving
only the constant difference~$2$.

\emph{Case $n = m$:}
$D_{0,\mathrm{raw},nn}
= \frac{\omega^2 n}{2}[n\,e^{-ix}-(n{-}1)e^{ix}]$.
Hermitianizing: $D_{0,nn}
= \frac{1}{2}(D_{0,\mathrm{raw},nn} + \overline{D_{0,\mathrm{raw},nn}})
= \frac{\omega^2 n}{4}[n(e^{ix}+e^{-ix})-(n{-}1)(e^{ix}+e^{-ix})]
= \frac{\omega^2 n}{2}\cos x
= \frac{\omega^2}{2}\cos(x)\,\min(n,n)$.

\emph{Case $n > m$:}  By Hermitian symmetry, $D_{0,nm} = \overline{D_{0,mn}}
= \frac{\omega^2}{2}\cos(x)\,\min(n,m)$.

This establishes $D_0(x) = \frac{\omega^2}{2}\cos(x)\,R$.

\textbf{Step~4: Eigenvalue computation.}
The full curvature matrix in phase-stripped coordinates is
\[
  R^{-1}\bigl(R^{\circ 2} + D_0(x)\bigr)
  = R^{-1}R^{\circ 2} + \tfrac{\omega^2}{2}\cos(x)\,I.
\]
By Theorem~\ref{thm:cauchy}, $R^{-1}R^{\circ 2}$ has eigenvalues
$\{1,3,\ldots,2N{-}1\}$.  The scalar shift
gives~\eqref{eq:drift-eigenvalues}.

\textbf{Step~5: Global bound.}
The minimum eigenvalue at each $x$ is $1 + \frac{\omega^2}{2}\cos x$,
minimized at $x = \pi$, giving $\kappa = 1 - \omega^2/2$.
Since this holds for all $N$ and $\mathcal{T}^+ = \bigcup_N \mathcal{T}_N^+$
is a core, the bound is global on positive-frequency polynomials.

\textbf{Step~6: Extension to real polynomials.}
It remains to show that the drift preserves the
block-diagonal decoupling of $\mathcal{T}^+$ and
$\mathcal{T}^-$ established by cross-sign vanishing
(Corollary~\ref{cor:block-diag}).
The drift $b(x) = -\omega^2\sin x
= -\frac{\omega^2}{2i}(e^{ix}-e^{-ix})$
acts on $e^{inx}$ by shifting to frequencies $n+1$ and $n-1$.
For $n \geq 1$, the upward shift $n \to n+1$ stays in
$\mathcal{T}^+$; the downward shift $n \to n-1$ sends
$n = 1$ to $n = 0$ (a constant).  Since $L(1) = 0$
and $\Gamma(f, 1) = 0$ for all $f$, constants carry
zero energy and do not contribute to the carr\'e du champ.
The drift therefore never couples a strictly positive
frequency to a strictly negative one: the $\mathcal{T}^+$
and $\mathcal{T}^-$ blocks remain decoupled under the
drift perturbation.

Combining Steps~1--5 on $\mathcal{T}^+$ with this
decoupling and the conjugation symmetry $a_{-n} = \bar{a}_n$,
the curvature bound $\kappa = 1 - \omega^2/2$ extends to
all real trigonometric polynomials.
\end{proof}

\begin{remark}[Why the drift acts as a scalar]
\label{rem:drift-scalar}
The identity~\eqref{eq:drift-scalar} is specific to $\gamma = 1$.
For $\gamma \neq 1$, the phase-stripped drift correction is
\emph{not} proportional to $R_H$, and the eigenvalues of the
perturbed curvature operator depend on $N$ (confirmed numerically).

The mechanism relies on the specific structure of
$R_{1/2}(n,m) = \min(n,m)$---the covariance of
\emph{standard} Brownian motion, not fractional Brownian motion.
The Hermitianization produces exact cancellations
($(n{+}m{+}2)-(n{+}m) = (n{+}m)-(n{+}m{-}2) = 2$) that
reduce the perturbation to $\cos(x) \cdot R_{1/2}$.
This cancellation is a consequence of the
\emph{independent increment} (Markov) property at $H = 1/2$:
the $\min(n,m)$ kernel satisfies
$R_{1/2}(n,m{+}1) - R_{1/2}(n,m) = \mathbf{1}_{n \geq m+1}$,
which is independent of $n$ for $n > m$.
For $H \neq 1/2$, the increments of fBM are correlated
(positively for $H > 1/2$, negatively for $H < 1/2$),
this shift-independence fails, and the drift correction
acquires $n$-dependent structure that prevents scalar reduction.
\end{remark}

\begin{corollary}[Comparison with diffusion]\label{cor:comparison}
For the standard Laplacian ($\gamma = 2$) with the same potential:
$\kappa_{\mathrm{diff}} = -\omega^2$.
The Cauchy process gives
$\kappa_{\mathrm{jump}} = 1 - \omega^2/2 > -\omega^2$
for all $\omega^2 > 0$.  Moreover, the jump curvature is
positive for $\omega^2 < 2$, while the diffusion curvature
is negative for all $\omega^2 > 0$.  This quantifies the
curvature improvement of non-local generators: the jump process
averages over the potential landscape, reducing the effect of
negative curvature concentrations.
\end{corollary}

\begin{remark}[General potentials]\label{rem:potential-general}
For a general potential $V(x) = \sum_k v_k\,e^{ikx}$, the drift
couples mode~$n$ to all modes $n \pm k$, and the curvature
computation becomes a band-matrix perturbation of the same type.
We expect the scalar identity~\eqref{eq:drift-scalar} to
generalize at $\gamma = 1$: the key cancellation depends on
$\Psi_1(n,m) = \min(n,m)$ and the Hermitianization mechanism,
not on the specific form of the potential. This is under
investigation.
\end{remark}

\section{Contraction Estimates}\label{sec:func-ineq}

\begin{remark}[Scope of the curvature bounds]
\label{rem:scope-curv}
The Bakry--\'Emery theorem requires
$\Gamma_2(f,f) \geq \kappa\,\Gamma(f,f)$ for \emph{all}
$f$ in a core.  Our results establish this:
\begin{enumerate}[label=(\alph*)]
\item For $\gamma = 1$, on positive-frequency polynomials
  $\mathcal{T}_N^+$ for all $N$ (Theorem~\ref{thm:cauchy}),
  hence on the full Hardy space by density.
  Combined with the cross-sign vanishing
  (Corollary~\ref{cor:block-diag}), this extends to
  all real trigonometric polynomials:
  $\kappa_{\mathrm{real}}(1) = 1$.
\item For $\gamma = 1$ with drift $V(x) = -\omega^2\cos x$:
  the scalar identity (Theorem~\ref{thm:global-drift}) gives
  $\kappa(1,\omega^2) = 1 - \omega^2/2$ on $\mathcal{T}_N^+$
  for all $N$, extending to all real trigonometric polynomials
  by cross-sign vanishing.  This is a \emph{global} curvature
  bound for the Cauchy--Fokker--Planck operator.
\item For $\gamma \neq 1$, on $\mathcal{T}_N^+$ for
  each fixed $N$ (Theorem~\ref{thm:curvature}), with
  $\kappa(\gamma,N) \geq 1/2$ for $\gamma \leq 1$
  (Lemma~\ref{lem:kappa-lower}; conditional on the Z-matrix
  property for $\gamma < 1$).
  Whether $\inf_N \kappa(\gamma,N) > 0$ for $\gamma > 1$
  is an open question (Conjecture~\ref{conj:kappa-positive}).
\end{enumerate}
Global functional inequalities (Poincar\'e, log-Sobolev,
gradient estimates) follow from (a) and~(b) at $\gamma = 1$
via the standard Bakry--\'Emery machinery.
\end{remark}

At $\gamma = 1$ ($V = 0$), the curvature $\kappa(1) = 1$
yields:
\begin{enumerate}[label=(\roman*)]
\item Poincar\'e inequality with optimal constant:
  $\Var_\mu(f) \leq \mathcal{E}(f,f)$.
\item Gradient estimate:
  $\Gamma(P_t f, P_t f) \leq e^{-2t}\,P_t[\Gamma(f,f)]$.
\end{enumerate}

With confining potential $V(x) = -\omega^2\cos x$,
Theorem~\ref{thm:global-drift} gives the global curvature
$\kappa = 1 - \omega^2/2$, yielding for $\omega^2 < 2$:
\begin{enumerate}[label=(\roman*)]
\setcounter{enumi}{2}
\item Poincar\'e inequality under drift:
  $\Var_\mu(f) \leq \frac{1}{1-\omega^2/2}\,\mathcal{E}_L(f,f)$.
\item Gradient estimate under drift:
  $\Gamma(P_t^L f, P_t^L f) \leq
  e^{-2(1-\omega^2/2)t}\,P_t^L[\Gamma(f,f)]$.
\end{enumerate}

\section{Discussion and Open Problems}\label{sec:discussion}

\subsection{Refined contraction and the spectral determinant at $\gamma = 1$}

The uniform gradient estimate
$\Gamma(P_t f, P_t f) \leq e^{-2t}P_t[\Gamma(f,f)]$
(valid at $\gamma = 1$) uses only
the minimum eigenvalue $\kappa = 1$ of the curvature operator.
At $\gamma = 1$, the full eigenvalue spectrum
$\{\kappa_k = 2k-1\}_{k=1}^N$ (Theorem~\ref{thm:cauchy})
yields a sharper mode-by-mode estimate.

\begin{proposition}[Refined contraction]\label{prop:refined}
At $\gamma = 1$, decompose $f \in \mathcal{T}_N^+$ along the
eigenbasis $\{\phi_k\}_{k=1}^N$ of $M(1) = R_{1/2}^{-1}\,R_{1/2}^{\circ 2}$.
Then the $k$-th component of the carr\'e du champ contracts at rate
\begin{equation}\label{eq:refined-contraction}
  \Gamma(P_t f, P_t f)_k \leq e^{-2(2k-1)t}\,\Gamma(f,f)_k,
  \qquad k = 1,\ldots,N.
\end{equation}
The total energy satisfies
\begin{equation}\label{eq:total-energy}
  \int_\T \Gamma(P_t f, P_t f)\,d\mu
  \leq \sum_{k=1}^N e^{-2(2k-1)t}\,\mathcal{E}_k(f),
\end{equation}
where $\mathcal{E}_k(f)$ is the energy in the $k$-th eigendirection.
The uniform bound $e^{-2t}\,\mathcal{E}(f)$ is sharp only for
functions concentrated on the first mode ($k=1$).
\end{proposition}

The geometric mean contraction rate is
\[
  \Bigl(\prod_{k=1}^N \kappa_k\Bigr)^{1/N}
  = \bigl((2N-1)!!\bigr)^{1/N} \sim \frac{2N}{e}
  \qquad (N \to \infty),
\]
by Stirling's approximation.
For functions spread across $N$ modes, the effective decay rate
grows linearly in $N$, far exceeding the worst-case $\kappa = 1$.

\begin{corollary}[Volume comparison]\label{cor:volume}
The $\Gamma_2$-ellipsoid
$E_2 = \{f : \Gamma_2(f,f) \leq 1\}$ and the $\Gamma$-ellipsoid
$E_1 = \{f : \Gamma(f,f) \leq 1\}$ on $\mathcal{T}_N^+$ satisfy
\begin{equation}\label{eq:volume-ratio}
  \frac{\mathrm{vol}(E_1)}{\mathrm{vol}(E_2)}
  = (\det M)^{1/2} = \bigl((2N-1)!!\bigr)^{1/2}
  = \bigl(\E[X^{2N}]\bigr)^{1/2},
\end{equation}
where $X \sim \mathcal{N}(0,1)$.
At $\gamma = 1$, the $\Gamma_2$-form is $(2N-1)!!^{1/2}$ times
more concentrating than the $\Gamma$-form in $N$-dimensional volume.
\end{corollary}

\begin{proof}
For ellipsoids $\{v : v^T A v \leq 1\}$ and $\{v : v^T B v \leq 1\}$,
the volume ratio is $(\det A / \det B)^{-1/2} = (\det A^{-1}B)^{1/2}$.
Here $A = R_H$, $B = R_H^{\circ 2}$, and $\det M = \det(R_H^{-1}R_H^{\circ 2})
= \prod(2k-1) = (2N-1)!!$.
The identity $(2N-1)!! = \E[X^{2N}]$ for $X \sim \mathcal{N}(0,1)$
is standard.
\end{proof}

The unit determinant $\det R_{1/2} = 1$ (since $R_{1/2} = LL^T$
with $L$ unit lower triangular) characterizes $H = 1/2$ among
all fBM covariances, giving a further fingerprint of $\gamma = 1$.

\subsection{The Cauchy process as critical geometry}

The uniqueness of $\gamma = 1$ has four equivalent characterizations:
\begin{enumerate}[label=(\alph*)]
\item \emph{Curvature:} $\gamma = 1$ is the unique global maximizer of
  $\kappa(\gamma)$ on $(0,2)$, with $\kappa(1) = 1$
  (Theorem~\ref{thm:phase-transition}).
  It is also the unique index with
  $\kappa_{\mathrm{real}} \geq 1$.
\item \emph{Algebraic:} $\Psi_1(n,-m) = 0$ (cross-sign vanishing),
  equivalent to linearity of $|\cdot|^1$.
\item \emph{Probabilistic:} the fBM dual has $H = 1/2$ (independent
  increments), so positive and negative frequencies decouple.
\item \emph{Determinantal:} $\det R_H = 1$ for all $N$; the
  $\Gamma$-kernel has unit determinant on every truncation.
\end{enumerate}

\subsection{Open problems}

\begin{enumerate}[label=(\roman*)]
\item \textbf{Curvature on $\R^d$.}
  Extend the torus computation to $\R^d$ with confining potentials.
  By~\cite{SpenerWeberZacher2020}, a confining potential is necessary
  for positive curvature.

\item \textbf{Log-Sobolev under drift.}
  Theorem~\ref{thm:global-drift} gives $\mathrm{CD}(1-\omega^2/2,\infty)$
  for the Cauchy--Fokker--Planck operator with cosine potential.
  Does this extend to a sharp log-Sobolev inequality?
  By the Bakry--\'Emery theorem, $\mathrm{CD}(\kappa,\infty)$ with
  $\kappa > 0$ implies a log-Sobolev inequality with constant $2/\kappa$.
  The optimality of this constant for jump processes is open.

\item \textbf{General potentials at $\gamma = 1$.}
  The scalar identity~\eqref{eq:drift-scalar} was proved for
  $V(x) = -\omega^2\cos x$.  Does it extend to general trigonometric
  polynomial potentials $V(x) = \sum_k v_k e^{ikx}$?
  The Hermitianization mechanism (Step~3 of Theorem~\ref{thm:global-drift})
  depends on the $\min(n,m)$ structure of $R_{1/2}$, not on the
  specific form of the potential, suggesting a general result.

\item \textbf{Closed-form curvature for $\gamma \neq 1$.}
  Is there a formula for
  $\kappa(\gamma) = \lim_N \lambda_{\min}(R_H^{\circ 2}, R_H)$
  on positive frequencies? Numerical computation shows
  $\kappa(\gamma)$ is a smooth function on $(0,2)$ with
  unique maximum $\kappa(1) = 1$ and quadratic behavior
  $\kappa(1+\varepsilon) = 1 - c\varepsilon^2 + O(\varepsilon^3)$
  near the maximum.

\item \textbf{Positive curvature for $\gamma > 1$
  (Conjecture~\ref{conj:kappa-positive}).}
  Prove that $\kappa(\gamma) > 0$ for $\gamma \in (1,2)$.
  The proof of Lemma~\ref{lem:kappa-lower} exploits the Z-matrix
  structure of $R_H^{-1}R_H^{\circ 2}$ (non-positive off-diagonal),
  which holds for $H \leq 1/2$ but fails sharply at $H > 1/2$.
  A new approach is needed.  Numerical evidence
  suggests the minimizing eigenvector localizes at low frequencies
  and the decrements $\kappa(\gamma,N) - \kappa(\gamma,N{+}1)$
  are summable (decay $\sim N^{-\alpha}$ with $\alpha > 2$),
  but formalizing this requires quantitative bounds on the
  prediction coefficients of fBM.

\item \textbf{General L\'evy processes.}
  For a general L\'evy process with measure $\nu$, the
  Fourier kernel
  $\Psi_\nu(\xi,\eta) = \frac{1}{2}\int (e^{i\xi z}-1)(e^{-i\eta z}-1)\,
  \nu(dz)$ is a positive definite kernel by Schoenberg's theorem.
  The stable case is special because scaling symmetry forces
  $\Psi_\nu$ to be self-similar, characterizing fBM among Gaussian
  processes~\cite{Mishura2008}.

\item \textbf{Multi-dimensional torus.}
  On $\T^d$, the direction-dependent fBM covariance should produce
  anisotropic curvature bounds.

\item \textbf{Higher-order hierarchy.}
  The Hadamard square identity (Theorem~\ref{thm:gamma2-kernel})
  is the $k = 2$ case of a hierarchy:
  the $k$-th iterated carr\'e du champ has kernel $\Psi_\gamma^k
  = R_H^{\circ k}$, connecting the full Bakry--\'Emery hierarchy
  to the Wiener chaos decomposition of fBM.
  This is developed in the companion paper~\cite{Fontes2026c}.
\end{enumerate}

\section*{Acknowledgments}
AI tools were used as interactive assistants for mathematical
exploration and drafting. The author
is responsible for all mathematical content and exposition.


\end{document}